\documentclass[a4paper,12pt,reqno]{article}
\usepackage{amsmath,amssymb,amsthm}
\usepackage{enumerate}
\usepackage{pxfonts}
\usepackage{color,graphicx}
\usepackage{authblk}
\usepackage{mathrsfs}
\usepackage{latexsym}

\usepackage[top=1truein,bottom=1truein,left=1truein,right=1truein]{geometry}
\usepackage{setspace}
\usepackage{geometry}

\usepackage[numbers,sort&compress]{natbib}
\usepackage{upgreek}
\usepackage{hyperref}

\newtheorem{theorem}{Theorem}[section]

\newtheorem{lemma}[theorem]{Lemma}

\newtheorem{remark}[theorem]{Remark}
\newtheorem{assumption}[theorem]{Assumption}

\newcommand{\var}{\mathrm{Var}}
\newcommand{\diag}{\mathrm{diag}}

\newcommand{\argmin}{\mathop{\rm arg~min}\limits}

\makeatletter
\@addtoreset{equation}{section}

\makeatother
\title{Quasi-Likelihood Analysis of Fractional Brownian Motion with Constant Drift under High-Frequency Observations}
\author
{Tetsuya Takabatake\footnote{tkbtk@hiroshima-u.ac.jp}}
\affil
{School of Economics, Hiroshima University\footnote{2-1 Kagamiyama 1-Chome, Higashi-Hiroshima, Hiroshima, Japan}}
\date{\today}

\begin{document}
\maketitle

\begin{abstract}
	Consider an estimation of the Hurst parameter $H\in(0,1)$ and the volatility parameter $\sigma>0$ for a fractional Brownian motion with a drift term 
	under high-frequency observations with a finite time interval. 
	In the present paper, we propose a consistent estimator of the parameter $\theta=(H,\sigma)$ combining the ideas of a quasi-likelihood function based on a local Gaussian approximation of a high-frequently observed time series and its frequency-domain approximation.	
	Moreover, we prove an asymptotic normality property of the proposed estimator for all $H\in(0,1)$ when the drift process is constant. 
\end{abstract}
\vspace{0.2cm}

\section{Introduction}
Let $(\Omega,\mathcal{F},P)$ be a complete probability space. 
Consider the 
stochastic process $X^{\theta}=\{X^{\theta}_{t}\}_{t\in[0,1]}$ defined on $(\Omega,\mathcal{F},P)$ of the form
\begin{equation}\label{def_X}
 \mathrm{d}X_{t}^{\theta}=\mu_{t}\,\mathrm{d}t+\sigma\,\mathrm{d}B_{t}^{H},\ \ X_{0}^{\theta}=\xi_{0},
 \ \ \theta=(H,\sigma)\in(0,1)\times(0,\infty),
\end{equation}
where 
$B^{H}$ is a fractional Brownian motion~(fBm) with Hurst parameter 
$H$, $\{\mu_{t}\}_{t\in[0,1]}$ is a continuous stochastic process and $\xi_{0}$ is a random variable. 
The stochastic process in (\ref{def_X}) is, for example, used for 
the log-volatility process and such volatility models recently attract much attention from researchers in mathematical finance and financial econometrics and practitioners in the financial industry,~e.g.~see 
\cite{Gatheral-Jaisson-Rosenbaum-2018}, \cite{Bayer-Friz-Gatheral-2016} and \cite{Fukasawa-Takabatake-Westphal-2022+} for details. 
The aim of the present paper is to investigate an asymptotic distribution of a quasi-likelihood-type estimator of the parameter $\theta$ 
based on the high-frequency data $\mathbf{X}^{\theta}_{n}:=(X^{\theta}_{0},X^{\theta}_{\delta_{n}},\cdots,X^{\theta}_{n\delta_{n}})$ with $\delta_{n}:=1/n$ when the sample size $n$ goes to infinity. 

First, note that $\theta$ is the only 
parameter to be estimated in (\ref{def_X}) because 
we can not consistently estimate drift parameters even if a sample path of $\{X_t\}_{t\in[0,1]}$ is continuously observed. 
Under high-frequency asymptotics, i.e.~$\delta_{n}\to 0$ as $n\to 0$, we can show
\begin{equation}\label{local_Gaussian_app1}
	X_{j\delta_{n}}^\theta-X_{(j-1)\delta_{n}}^\theta
	=\sigma(B_{j\delta_{n}}^{H}-B_{(j-1)\delta_{n}}^{H})+O_{P}(\delta_{n})\ \ \mbox{as $n\to\infty$}
\end{equation}
and, thanks to the self-similarity property of the fBm, we have
\begin{equation}\label{local_Gaussian_app2}
	\sigma(B_{j\delta_{n}}^{H}-B_{(j-1)\delta_{n}}^{H})
	\stackrel{\mathcal{L}}{=}\sigma\delta_{n}^{H}(B_{j}^{H}-B_{j-1}^{H}),\ \ j=1,\cdots,n,
\end{equation}
where $\stackrel{\mathcal{L}}{=}$ means that the equality holds in law. 
If we know the Hurst parameter $H=1/2$, 
then the volatility $\sigma$ is the only parameter to be estimated and it is well-known that 
the QMLE~(Quasi-Maximum Likelihood Estimator) of $\sigma$ based on the local Gaussian approximation~(\ref{local_Gaussian_app1}), 
which is same as the quadratic variation of $X^{\theta}$ with equidistant sampling $t_{j}^{n}:=j/n$ since the fBm has the independent increments property when $H=1/2$, 
is consistent and asymptotically normal 
as $n\to\infty$ under some mild technical assumptions of $\{\mu_{t}\}_{t\in[0,1]}$, e.g.~see \cite{Fukasawa-2010}.
Then the drift term can be seen as a nuisance parameter because 
the QMLE of $\sigma$ 
can be computed 
without identifying the drift term and its asymptotic distribution does not depend on the drift term. 

Thanks to (\ref{local_Gaussian_app1}) and (\ref{local_Gaussian_app2}), 
even if 
$H\in(0,1)$ is unknown, 
it would be possible to 
consistently estimate the parameter $\theta$ without identifying the drift term under high-frequency asymptotics 
using a 
quasi-likelihood-type estimator based on the local Gaussian approximation (\ref{local_Gaussian_app1}). 
On the other hand, it is unclear whether an asymptotic distribution of the 
quasi-likelihood-type estimator does not depend on the drift term because (\ref{local_Gaussian_app1}) and (\ref{local_Gaussian_app2}) imply that it becomes more difficult to distinguish between the noise and drift terms of 
$X^{\theta}$ from high-frequency data when $H$ approaches $1$.

Recently, \cite{Gairing-Imkeller-Shevchenko-Tudor-2020} proposed
an estimator of the parameter $\theta$ using the change-of-frequency method 
and proved its asymptotic normality property for all $H\in(0,H_{+}]$ with $H_{+}\in(0,1)$ under the technical condition $\delta_{n}:=n^{-\alpha}$ for some $\alpha\geq 1$ satisfying $\alpha>\frac{2H_{+}-1}{2(1-H_{+})}\geq 1$ if $H_{+}\geq 3/4$. 
We remark that the condition $\alpha>1$ assumed in the case $H_{+}\geq 3/4$ 
is not standard because it implies the length of the observation period $T_{n}:=n\delta_{n}=n^{1-\alpha}$ converges to zero. 
Therefore, it is not trivial whether their proposed estimator enjoys the same asymptotic normality property when $H\geq 3/4$ under the condition $\delta_{n}:=1/n$. 
Moreover,
their proposed estimator is, of course, not 
optimal because the covariance structure of the noise is not used in their estimation procedure. 

In the present paper, we propose an estimator of the parameter $\theta$ combining the following two ideas in the similar way to \cite{Fukasawa-Takabatake-Westphal-2022+}:~(1)~the local Gaussian approximation~(\ref{local_Gaussian_app1}) and (2)~the frequency domain approximation, the so-called Whittle approximation, of the quasi-likelihood function, see Section~\ref{Subsection_QWLE} for details.  
Then we can easily prove the consistency of the proposed estimator in the similar way to Theorem~2.8 of \cite{Fukasawa-Takabatake-Westphal-2022+}. 
Our contribution in the present paper is to prove that $(1)$~the proposed estimator enjoys the asymptotic normality property for all $H\in(0,1)$ even when $\delta_{n}:=1/n\to 0$ as $n\to\infty$, 
which implies the technical condition of $\delta_{n}$ assumed in \cite{Gairing-Imkeller-Shevchenko-Tudor-2020} is not essential to derive asymptotic distributions of estimators of $\theta$, and $(2)$~an asymptotic distribution of the proposed estimator does not depend on the drift term for all $H\in(0,1)$, at least, 
when $\{\mu_{t}\}_{t\in[0,1]}$ is . 
constant. 

The present paper is organized as follows. 
We summarize preliminary results and notation in Section~\ref{Section_Preliminary_Results}. 
In Section~\ref{Section_Main_Result}, our proposed estimator is defined and a main theorem in the present paper is given. 
The main theorem is proven in Section~\ref{Section_Proof_Main_Theorem} and a preliminary lemma used in the proof of the main theorem is proven  in Section~\ref{Section_Proof_Key_Lemma}.   

\section{Preliminary Results}\label{Section_Preliminary_Results}
\subsection{Fractional Brownian Motion}
A centered continuous Gaussian process $\{B_t^H\}_{t\in\mathbb{R}}$, defined on a complete probability space $(\Omega,\mathcal{F},P)$, is called a fractional Brownian motion (fBm) with Hurst parameter $H\in(0,1]$ if $B_{0}^{H}=0$ $P$-a.s. and it satisfies the following scaling property:
\begin{equation}\label{fBm_scaling}
	E[|B_{t}^{H}-B_{s}^{H}|^{2}]=|t-s|^{2H}\ \ \mbox{for any $s, t\in\mathbb{R}$}.
\end{equation}
From (\ref{fBm_scaling}), it is obvious that the fBm has the stationary increments and self-similarity properties.  
Moreover, it is well-known that the spectral density function of the stationary increments process $\{B_{t}^{H}-B_{t-1}^{H}\}_{t\in\mathbb{Z}}$ 
is given by
\begin{equation}\label{spd_fGn}
	f_{H}(\lambda):=C_H\{2(1-\cos{\lambda})\}\sum_{j\in\mathbb{Z}}\frac{1}{|\lambda+2\pi j|^{1+2H}},\ \ \lambda\in[-\pi,\pi],
\end{equation}
with $C_H:=(2\pi)^{-1}\Gamma(2H+1)\sin(\pi H)$, e.g. see 
\cite{Samorodnitsky-Taqqu-1994}.
\subsection{Notation}
Consider the parameter space 
$\Theta:=\Theta_{H}\times(0,\infty)$, where $\Theta_{H}$ is a compact set of $(0,1)$. 
Denote by $\theta_{0}=(H_{0},\sigma_{0})$ the true value of the parameter $\theta=(H,\sigma)$. 
Let $n\in\mathbb{N}$ be the sample size and $\delta_{n}:=1/n$ be the length of sampling intervals. 
Denote by $\mathbf{1}_{n}:=(1,\cdots,1)\in\mathbb{R}^{n}$ and
\begin{align*}
	&\Delta\mathbf{X}_{n}^{\theta}:=(X_{\delta_{n}}^{\theta}-X_{0}^{\theta}, X_{2\delta_{n}}^{\theta}-X_{\delta_{n}}^{\theta}, \cdots, X_{n\delta_{n}}^{\theta}-X_{(n-1)\delta_{n}}^{\theta}),\\ 
	&\Delta\mathbf{B}_{n}^{\theta}:=(\sigma B_{\delta_{n}}^{H}, \sigma(B_{2\delta_{n}}^{H}-B_{\delta_{n}}^{H}), \cdots, \sigma(B_{n\delta_{n}}^{H}-B_{(n-1)\delta_{n}}^{H})).
\end{align*}
Moreover, set $\Delta\widetilde{\mathbf{X}}_{n}:=\delta_{n}^{-H_{0}}b(H_{0})^{-\frac{1}{2}}\Delta\mathbf{X}_{n}^{\theta_{0}}$ and $\Delta\widetilde{\mathbf{B}}_{n}:=\delta_{n}^{-H_{0}}b(H_{0})^{-\frac{1}{2}}\Delta\mathbf{B}_{n}^{\theta_{0}}$.

	Let $\tau\in\mathbb{Z}$. For an integrable function $f:[-\pi,\pi]\to[-\infty,\infty]$, 
	the $\tau$-th Fourier coefficient of $f$ is defined by
	\begin{equation*}
		\widehat{f}(\tau):=\int_{-\pi}^{\pi}e^{\sqrt{-1}\tau x}f(x)\,\mathrm{d}x.
	\end{equation*}	
	We denote by $T_n(f)$ the $n\times n$-Toeplitz matrix whose $(i,j)$-element is given by $\widehat{f}(i-j)$ for each $i,j=1,\cdots,n$. 
Thanks to the self-similarity property of the fBm, we can write 
$\var[\Delta\mathbf{B}_{n}^{\theta}]=T_{n}(f_{\theta}^{n})$ with $f_{\theta}^{n}(\lambda):=\sigma^{2}\delta_{n}^{2H}f_{H}(\lambda)$. 
Define by 
\begin{align*}
	&b(H):=\exp\left(\frac{1}{2\pi}\int_{-\pi}^{\pi}\log{f_{H}(\lambda)}\,\mathrm{d}\lambda\right),\ \ g_{H}(\lambda):=b(H)^{-1}f_{H}(\lambda),\\
	&h_{H}(\lambda):=1/(4\pi^{2}g_{H}(\lambda)),\ \ \nu_{n}^2(H):=\frac{1}{n}\left\langle\Delta\mathbf{X}_{n}^{\theta_{0}},T_{n}(h_{H})\Delta\mathbf{X}_{n}^{\theta_{0}}\right\rangle_{\mathbb{R}^{n}}.
\end{align*}
Set $h_{H}^{(j)}(\lambda):=(\partial/\partial{H})^{j}h_{H}(\lambda)$ and $g_{H}^{(j)}(\lambda):=(\partial/\partial{H})^{j}g_{H}(\lambda)$ for $j=0,1$. 
Finally, $\stackrel{\mathcal{L}}{\to}$ denotes the convergence in law under the probability measure $P$.
\section{Main Result}\label{Section_Main_Result}
\subsection{Quasi-Whittle Likelihood Estimator~(QWLE)}\label{Subsection_QWLE}
In the present paper, we 
propose the estimator $\widehat{\theta}_{n}:=(\widehat{H}_{n},\widehat{\sigma}_{n})$ defined by
\begin{equation}\label{definition_QWLE}
	\widehat{H}_{n}:=\argmin_{H\in\Theta_H}\nu_{n}^{2}(H),\ \ 
	\widehat{\sigma}_{n}:=\delta_{n}^{-\widehat{H}_{n}}b(\widehat{H}_{n})^{-\frac{1}{2}}\sqrt{\nu_{n}^{2}(\widehat{H}_{n})}.
\end{equation}
We call $\widehat{\theta}_{n}$ the QWLE~(Quasi-Whittle Likelihood Estimator) in the following. 
In the rest of this subsection, we make several remarks on the proposed estimator 
in order.
\begin{remark}[\it Local Gaussian Approximation and Whittle Approximation]\rm
Thanks to (\ref{local_Gaussian_app1}) and (\ref{local_Gaussian_app2}), 
it would be possible that 
$\Delta{\mathbf{X}}_{n}^{\theta}$ is approximated by the Gaussian vector $\Delta{\mathbf{B}}_{n}^{\theta}$ in some suitable sense under high-frequency asymptotics. 
Therefore, 
we use the likelihood function of $\Delta{\mathbf{B}}_{n}^{\theta}$ as a quasi-likelihood function of $\Delta{\mathbf{X}}_{n}^{\theta}$. 
Actually, we utilize 
an approximate likelihood function of $\Delta{\mathbf{B}}_{n}^{\theta}$ in the frequency domain, the so-called Whittle likelihood function, as a quasi-likelihood function of $\Delta{\mathbf{X}}_{n}^{\theta}$ because the quasi-maximum likelihood estimator is computationally infeasible when the sample size $n$ is quite large due to the high computational cost of the inverse and determinant of $\mathrm{Var}[\Delta{\mathbf{B}}_{n}^{\theta}]$. 

In the rest of this remark, 
we explain why the estimator $\widehat{\theta}_{n}$ is defined by (\ref{definition_QWLE}). 
Set $\nu_{n}\equiv\nu_{n}(\theta):=\sigma\delta_{n}^{H}b(H)^{\frac{1}{2}}$ and $\widehat{\nu}_{n}:=\{\nu_{n}^{2}(\widehat{H}_{n})\}^{1/2}$. 
	First note that $\nu_{n}^{2}g_{H}(\lambda)$ is the spectral density function of 
	the stationary Gaussian sequence $\{\sigma(B_{j\delta_{n}}^{H}-B_{(j-1)\delta_{n}}^{H})\}_{j\in\mathbb{Z}}$ with respect to the reparameterized parameter $(H,\nu_{n})$. 
	Since we have 
	\begin{equation}\label{key_repara_spd}
      \frac{1}{2\pi}\int_{-\pi}^{\pi}\log{g_{H}(\lambda)}\,\mathrm{d}\lambda=0    
    \end{equation}
    for all $H\in(0,1)$,  
    the quasi-Whittle likelihood function of 
    $\Delta\mathbf{X}_{n}^{\theta_{0}}$ with respect to the reparameterized parameter $(H,\nu_{n})$ is defined by
     \begin{equation*}
     	{L}_{n}^{(0)}(H,\nu_{n}):=\frac{1}{2}\left(\log{\nu_{n}^{2}}+\frac{1}{\nu_{n}^{2}}\nu_{n}^{2}(H)\right).
     \end{equation*}
     Since $(\widehat{H}_{n},\widehat{\nu}_{n})$ is 
     a minimizer of the quasi-Whittle likelihood function ${L}_{n}^{(0)}(H,\nu_{n})$ on $\Theta_{H}\times(0,\infty)$
     for each $n\in\mathbb{N}$,  
     the estimator $\widehat\sigma_{n}$ can be defined as (\ref{definition_QWLE}) using the estimator $(\widehat{H}_{n},\widehat{\nu}_{n})$ 
	and the relation $\nu_{n}=\sigma\delta_{n}^{H}b(H)^{\frac{1}{2}}$.     
\end{remark}
\begin{remark}[\textit{Reparameterization}]\rm
	Under high-frequency observations, the effects of $\sigma$ and $H$
	fuse in the limit and the asymptotic Fisher information matrix when a ``diagonal'' rate-matrix is used becomes
	singular due to the self-similarity
	property of the fractional Gaussian noise. 
	As a result, it is necessary to reparametrize the parameter
	$\sigma$ in order to obtain a limit theorem of an estimator. 
	See \cite{Brouste-Fukasawa-2018} and \cite{Fukasawa-Takabatake-2019} for more details.

	In the rest of this remark, we briefly explain 
	of our strategy to prove asymptotic properties of 
	the QWLE $\widehat{\theta}_{n}$. 
	First note that 
	$\widehat{H}_{n}$ is also a minimizer of the function
	\begin{equation*}
		\widetilde{\sigma}_n^2(H):=\delta_{n}^{-2H_{0}}b(H_{0})^{-1}\nu_{n}^{2}(H)
	\end{equation*}
	with respect to $H$ on $\Theta_{H}$ 
	so that the random variable $\widetilde{\theta}_{n}:=(\widehat{H}_{n},\widetilde{\sigma}_{n})$ with $\widetilde{\sigma}_{n}:=\{\widetilde{\sigma}_{n}^{2}(\widehat{H}_{n})\}^{1/2}$ is a minimizer of the function
	\begin{equation}\label{pseudo-WL}
       	L_{n}(\theta):=\frac{1}{2}\left(\log{\sigma^2}+\frac{1}{\sigma^2}\widetilde{\sigma}_n^2(H)\right)
       	=\frac{1}{2}\left(\log{\sigma^2}+\frac{1}{\sigma^{2}n}\left\langle\Delta\widetilde{\mathbf{X}}_{n},T_{n}(h_{H})\Delta\widetilde{\mathbf{X}}_{n}\right\rangle_{\mathbb{R}^{n}}\right)
    \end{equation}
    with respect to $\theta=(H,\sigma)$ on $\Theta_{H}\times(0,\infty)$. 
    Note that $L_{n}(\theta)$ is not an estimation function since the true value $H_{0}$ is used in its definition. 
    It plays, however, a similar role to the usual estimation function in  
	proofs of asymptotic properties of the QWLE 
    because 
    $L_{n}(\theta)$ is the quasi-Whittle likelihood function of the suitably rescaled random vector $\Delta\widetilde{\mathbf{X}}_{n}$  
    and 
    the quasi-spectral density function $\sigma^{2}g_{H}(\lambda)$ 
    which appears in $L_{n}(\theta)$ no longer depends on the asymptotic parameter $n$. 
    Therefore, 
    it would be possible that the random variable $\widetilde{\theta}_{n}$ converges to $\theta_{0}$ in some suitable sense as $n\to\infty$ and asymptotic properties of the estimator $\widehat{\theta}_{n}$ can be proven using the convergence of $\widetilde{\theta}_{n}$. 
    See Section~5.2 of \cite{Fukasawa-Takabatake-2019}, Section~4 of \cite{Fukasawa-Takabatake-Westphal-2022+} and Section~\ref{section_proof_asynormal} for more details.
\end{remark}
\begin{remark}[\it Implementation]\rm
	In this remark, we briefly explain how to efficiently implement the QWLE. 
	First note that we can write 
	\begin{equation}\label{nu_H_integral}
		\nu_{n}^2(H)=\frac{1}{2\pi}\int_{-\pi}^{\pi}\frac{I_{n}(\lambda,\Delta\mathbf{X}_{n}^{\theta_{0}})}{g_{H}(\lambda)}\,\mathrm{d}\lambda,
	\end{equation}
	where $I_{n}(\lambda,x)$ is the periodogram defined by
	\begin{equation*}
		I_{n}(\lambda,x):=\frac{1}{2\pi n}\left|\sum_{j=1}^{n}x_{j}e^{\sqrt{-1}j\lambda}\right|,\ \ \lambda\in[-\pi,\pi],\ \ x=(x_{1},\cdots,x_{n})\in\mathbb{R}^{n}.
	\end{equation*}
	Then the Riemann approximation of the integral~(\ref{nu_H_integral}) 
	gives
	\begin{equation}\label{nu_H_sum}
		\nu_{n}^2(H)\approx\frac{1}{n}\sum_{j=1}^{n}\frac{I_{n}(\lambda_{j}^{n},\Delta\mathbf{X}_{n}^{\theta_{0}})}{g_{H}(\lambda_{j}^{n})},\ \ \lambda_{j}^{n}:=\frac{2\pi j}{n},
	\end{equation}
	and the sum in (\ref{nu_H_sum}) can be effectively computed using the fast Fourier transform algorithm. 
	Note that the series appears in the  
	function $g_{H}(\lambda)$ 
	can be accurately and efficiently computed using the approximation method proposed by \cite{Paxson-1997}. See also \cite{Fukasawa-Takabatake-2019} and its supplementary article~\cite{Fukasawa-Takabatake-2019-supplement} for more details.
\end{remark}
\subsection{Asymptotic Normality Property of QWLE}
First, we introduce a class of sequences of non-diagonal rate matrices which plays a key role to prove an asymptotic normality property of QWLE with a non-degenerate asymptotic variance-covariance matrix. 
\begin{assumption}\label{assumption_rate_matrix}
	Assume a sequence of matrices $\{\varphi_{n}(\theta)\}_{n\in\mathbb{N}}$ and a matrix $\overline{\varphi}(\theta)$ of the forms
	\begin{equation*}
		\varphi_{n}(\theta):=\frac{1}{\sqrt{n}}
		\begin{pmatrix}
		\varphi_{n}^{11}(\theta)&\varphi_{n}^{12}(\theta) \\
		\varphi_{n}^{21}(\theta)&\varphi_{n}^{22}(\theta)
		\end{pmatrix},\ \
		\overline{\varphi}(\theta):=
		\begin{pmatrix}
			\overline{\varphi}_{11}(\theta)&\overline{\varphi}_{12}(\theta)\\
			\overline{\varphi}_{21}(\theta)&\overline{\varphi}_{22}(\theta)
		\end{pmatrix}
	\end{equation*}
	satisfy the following properties for each $\theta\in\Theta$:
	\begin{enumerate}[$(1)$]
		\item $\varphi_{n}^{11}(\theta)\rightarrow\overline\varphi_{11}(\theta)$ as $n\to\infty$,
		\item $\varphi_{n}^{12}(\theta)\rightarrow\overline\varphi_{12}(\theta)$ as $n\to\infty$,
		\item $s_n^{21}(\theta):=\varphi_{n}^{11}(\theta)\sigma\log\delta_{n}+\varphi_{n}^{21}(\theta) \rightarrow\overline\varphi_{21}(\theta)$ as $n\to\infty$,
		\item $s_n^{22}(\theta):=\varphi_{n}^{12}(\theta)\sigma\log\delta_{n}+\varphi_{n}^{22}(\theta) \rightarrow\overline\varphi_{22}(\theta)$ as $n\to\infty$,
		\item $\varphi_{n}^{11}(\theta)\varphi_{n}^{22}(\theta)-\varphi_{n}^{12}(\theta)\varphi_{n}^{21}(\theta)\neq 0$ for each $n\in\mathbb{N}$,
		\item $\overline\varphi_{11}(\theta)\overline\varphi_{22}(\theta)-\overline\varphi_{12}(\theta)\overline\varphi_{21}(\theta)\neq 0$.
	\end{enumerate}
\end{assumption}
Then we can prove a main theorem in the present paper as follows.
\begin{theorem}\label{theorem_asy-normality_QWLE}
	Consider a sequence of rate matrices $\{\varphi_{n}(\theta)\}_{n\in\mathbb{N}}$ satisfying Assumption~$\ref{assumption_rate_matrix}$. 
	Assume $\theta_{0}=(H_{0},\sigma_{0})$ is an interior point of $\Theta$. 
	Then we obtain the following result:
	\begin{enumerate}[$(1)$]
		\item The sequence of the QWLEs $\{\widehat\theta_{n}\}_{n\in\mathbb{N}}$ is (weakly) consistent as $n\to\infty$.
		\item If $\{\mu_{t}\}_{t\in[0,1]}$ is identically equal to a 
		$\mathcal{F}$-measurable random variable $\mu$, then the sequence of the QWLEs $\{\widehat\theta_{n}\}_{n\in\mathbb{N}}$ 
	satisfies the following asymptotic normality property:
	\begin{equation}\label{asy-normality_QWLE}
		\varphi_n(\theta_{0})^{-1}(\widehat\theta_{n}-\theta_{0})
		\stackrel{\mathcal{L}}{\to}\mathcal{N}\left(0,\mathcal{I}(\theta_{0})^{-1}\right)\ \ \mbox{as $n\to\infty$},
	\end{equation}
	where $\mathcal{I}(\theta)$ is the positive definite matrix defined by
	\begin{align*}
		\mathcal{I}(\theta):=\overline{\varphi}(\theta)^{\ast}\mathcal{F}(\theta)\overline{\varphi}(\theta),\ \ \mathcal{F}(\theta):=\frac{1}{4\pi}\int_{-\pi}^{\pi}\left(\frac{\partial}{\partial{\theta}}\log{f_{\theta}}(\lambda)\right)
		\left(\frac{\partial}{\partial{\theta}}\log{f_{\theta}}(\lambda)\right)^{\ast}\,\mathrm{d}\lambda.
	\end{align*}
	\end{enumerate}
\end{theorem}
Several examples of $\{\varphi_{n}(\theta)\}_{n\in\mathbb{N}}$ satisfying Assumption~\ref{assumption_rate_matrix} can be found in \cite{Fukasawa-Takabatake-2019}. 
Particular choices of $\{\varphi_{n}(\theta)\}_{n\in\mathbb{N}}$ imply that the convergence rates of $\widehat{H}_{n}$ and $\widehat{\sigma}_{n}$ are $\sqrt{n}$ and $\sqrt{n}/\log{\delta_{n}}$ respectively. 
See \cite{Fukasawa-Takabatake-2019} for details. 
\begin{remark}\rm
	In the case $\mu=0$, i.e.~$X^{\theta}=\sigma B^{H}$, Theorem~3 of \cite{Fukasawa-Takabatake-2019} proved that the Whittle estimator, defined in the same way as (\ref{definition_QWLE}), has the same asymptotic distribution as $(\ref{asy-normality_QWLE})$ 
	under high-frequency asymptotics. 
	Therefore, Theorem~$\ref{theorem_asy-normality_QWLE}$~(2) 
	implies that the asymptotic distribution of the QWLE does not depend on the drift term, at least, when $\{\mu_{t}\}_{t\in[0,1]}$ is constant.  
	We will investigate asymptotic properties of the QWLE when $\{\mu_{t}\}_{t\in[0,1]}$ is not constant 
	in the future work. 
\end{remark}
\section{Proof of Theorem~\ref{theorem_asy-normality_QWLE}}\label{Section_Proof_Main_Theorem}
\subsection{Preliminary Lemma}
Before proving Theorem~\ref{theorem_asy-normality_QWLE}, we prepare the following lemma. 
\begin{lemma}\label{lemma_Whittle_app}
	For any $\epsilon>0$, $H\in(0,1)$ and $j=0,1$, we have 
	\begin{align}
		&\left\langle\mathbf{1}_{n},T_{n}(h_{H}^{(j)})\mathbf{1}_{n}\right\rangle_{\mathbb{R}^{n}}=o(n^{2(1-H)+\epsilon})\ \ \mbox{as $n\to\infty$},\label{lemma_Whittle_app1}\\ 
		&\left\langle\mathbf{1}_{n},T_{n}(h_{H}^{(j)})T_{n}(g_{H})T_{n}(h_{H}^{(j)})\mathbf{1}_{n}\right\rangle_{\mathbb{R}^{n}}=o(n^{2(1-H)+\epsilon})\ \ \mbox{as $n\to\infty$}. \label{lemma_Whittle_app2}
	\end{align}
\end{lemma}
The proof of Lemma~\ref{lemma_Whittle_app} is left to Section~\ref{Section_Proof_Key_Lemma}.  
\subsection{Proof of Theorem~\ref{theorem_asy-normality_QWLE}}\label{section_proof_asynormal}
First, note that the consistency of the QWLE can be proven in the similar way to the proof of Theorem~2.8 of \cite{Fukasawa-Takabatake-Westphal-2022+}. 
In the following, we prove only the asymptotic normality property of the QWLE. 
In the similar way to the proof of Theorem~2.12 of \cite{Fukasawa-Takabatake-Westphal-2022+} and the proof of Theorem~3 of \cite{Fukasawa-Takabatake-2019}, the asymptotic normality property of the QWLE follows once we have proven
\begin{equation}\label{key_asynormality}
	\sqrt{n}\nabla{L}_n(\theta_{0})
	\stackrel{\mathcal{L}}{\to}\mathcal{N}\left(0,\diag\left(\mathcal{G}(H_{0}),2\sigma_{0}^{-2}\right)\right)\ \ \mbox{as $n\to\infty$},
\end{equation}
where the function $L_{n}(\theta)$ is defined by (\ref{pseudo-WL}) and 
\begin{equation*}
	\mathcal{G}(H):=\frac{1}{4\pi}\int_{-\pi}^{\pi}\left|\frac{\partial}{\partial{H}}\log{g_{H}(\lambda)}\right|^{2}\,\mathrm{d}\lambda.
\end{equation*}
Now we introduce notation used in the proof. 
Define by 
\begin{align*}
	A_{n}^{1}(\theta,x):=
	\frac{1}{2\sigma^{2}\sqrt{n}}\left\langle x,T_{n}(h_{H}^{(1)})x\right\rangle_{\mathbb{R}^{n}},\ \ 
	A_{n}^{2}(\theta,x):=
	\frac{1}{\sigma^{2}\sqrt{n}}\left(\left\langle x,T_{n}(h_{H}^{(0)})x\right\rangle_{\mathbb{R}^{n}}
	-\sigma^{2}\right)
\end{align*}
for $x\in\mathbb{R}^{n}$. Set $Y_{n}^{j}(\theta)\equiv A_{n}^{j}(\theta,\Delta\widetilde{\mathbf{X}}_{n})$ and $Z_{n}^{j}(\theta)\equiv A_{n}^{j}(\theta,\Delta\widetilde{\mathbf{B}}_{n})$ for $j=1,2$.
By a straight-forward calculation, we can write
\begin{equation*}
	\sqrt{n}\nabla{L}_n(\theta)=(Y_{n}^{1}(\theta),-\sigma^{-1}Y_{n}^{2}(\theta)).
\end{equation*}
Moreover, in the similar way to the proof of Theorem~2 of \cite{Fox-Taqqu-1986}, we can prove
\begin{equation*}
	\left(Z_{n}^{1}(\theta_{0}),-\sigma_{0}^{-1}Z _{n}^{2}(\theta_{0})\right)
	\stackrel{\mathcal{L}}{\to}\mathcal{N}\left(0,\diag\left(\mathcal{G}(H_{0}),2\sigma_{0}^{-2}\right)\right)\ \ \mbox{as $n\to\infty$}.
\end{equation*}
Therefore, in order to prove (\ref{key_asynormality}), it suffices to prove
\begin{equation}\label{key_asynormality2}
	\left\langle\Delta\widetilde{\mathbf{X}}_{n},T_{n}(h_{H_{0}}^{(j)})\Delta\widetilde{\mathbf{X}}_{n}\right\rangle_{\mathbb{R}^{n}}
	=\left\langle\Delta\widetilde{\mathbf{B}}_{n},T_{n}(h_{H_{0}}^{(j)})\Delta\widetilde{\mathbf{B}}_{n}\right\rangle_{\mathbb{R}^{n}}
	+o_{P}(n^{\epsilon})\ \ \mbox{as $n\to\infty$}
\end{equation}
for any $\epsilon>0$ and $j=0,1$. 
We prove (\ref{key_asynormality2}) in the rest of the proof. 
Since $\Delta\widetilde{\mathbf{X}}_{n}=\mu\delta_{n}^{1-H_{0}}b(H_{0})^{-1}\mathbf{1}_{n}+\Delta\widetilde{\mathbf{B}}_{n}$ and $\delta_{n}=1/n$, we can write
\begin{align}
	&\left\langle\Delta\widetilde{\mathbf{X}}_{n},T_{n}(h_{H_{0}}^{(j)})\Delta\widetilde{\mathbf{X}}_{n}\right\rangle_{\mathbb{R}^{n}}
	-\left\langle\Delta\widetilde{\mathbf{B}}_{n},T_{n}(h_{H_{0}}^{(j)})\Delta\widetilde{\mathbf{B}}_{n}\right\rangle_{\mathbb{R}^{n}} \nonumber\\
	&=
	2\mu b(H_{0})^{-1}n^{H_{0}-1}\left\langle\mathbf{1}_{n},T_{n}(h_{H_{0}}^{(j)})\Delta\widetilde{\mathbf{B}}_{n}\right\rangle_{\mathbb{R}^{n}}
	+\mu^2b(H_{0})^{-2}n^{2(H_{0}-1)}\left\langle\mathbf{1}_{n},T_{n}(h_{H_{0}}^{(j)})\mathbf{1}_{n}\right\rangle_{\mathbb{R}^{n}}. \label{key_asynormality3}
\end{align}
Moreover, we have 
\begin{equation}\label{key_asynormality6}
	\mathcal{L}\left\{\left\langle\mathbf{1}_{n},T_{n}(h_{H_{0}}^{(j)})\Delta\widetilde{\mathbf{B}}_{n}\right\rangle_{\mathbb{R}^{n}}\bigl|P\right\}
	\sim\mathcal{N}\left(0,\left\langle\mathbf{1}_{n},T_{n}(h_{H_{0}}^{(j)})T_{n}(g_{H_{0}})T_{n}(h_{H_{0}}^{(j)})\mathbf{1}_{n}\right\rangle_{\mathbb{R}^{n}}\right).
\end{equation}
Therefore (\ref{key_asynormality2}) follows from 
(\ref{key_asynormality3}), (\ref{key_asynormality6}) and Lemma~\ref{lemma_Whittle_app}. This completes the proof.
\section{Proof of Lemma~\ref{lemma_Whittle_app}}\label{Section_Proof_Key_Lemma}
\subsection{Notation}
Suppose $A$ is a real-valued $n\times n$-matrix. Define the operator norm of $A$ by
\begin{equation*}
	\|A\|_{\mathrm{op}}:=\sup_{x\in\mathbb{R}^{n}}\frac{\|Ax\|_{\mathbb{R}^{n}}}{\|x\|_{\mathbb{R}^{n}}}
\end{equation*}
and the Frobenius norm of $A$ by
\begin{equation*}
	\|A\|_{F}:=\left(\mathrm{Tr}\left[AA^{\ast}\right]\right)^{\frac{1}{2}}.
\end{equation*} 
In the present paper, we use the following well-known properties:
\begin{enumerate}[$(1)$] 
	\item $\|AB\|_{\mathrm{op}}\leq\|A\|_{\mathrm{op}}\|B\|_{\mathrm{op}}$. 
	\item $\|Ax\|_{\mathbb{R}^{n}}\leq\|A\|_{\mathrm{op}}\|x\|_{\mathbb{R}^{n}}$ and $\|A\|_{\mathrm{op}}\leq\|A\|_{F}$. 
\end{enumerate}
\subsection{Preliminary Lemma}
Before proving Lemma~\ref{lemma_Whittle_app}, we prove the following preliminary lemma.
\begin{lemma}\label{ext_lemma5.2_D89}
	For any $\epsilon>0$ and $H\in(0,1)$, 
	\begin{equation*}
		\left\|I_{n}-T_{n}(g_{H})^{\frac{1}{2}}T_{n}(h_{H})T_{n}(g_{H})^{\frac{1}{2}}\right\|_{F}=o(n^{\epsilon})\ \ \mbox{as $n\to\infty$.}
	\end{equation*}
\end{lemma}
\begin{proof}
	First note that we can write
	\begin{align*}
		&\left\|I_{n}-T_{n}(g_{H})^{\frac{1}{2}}T_{n}(h_{H})T_{n}(g_{H})^{\frac{1}{2}}\right\|_{F}^{2}\\
		&=n-2\mathrm{Tr}\left[T_{n}(g_{H})T_{n}(h_{H})\right]+\mathrm{Tr}\left[T_{n}(g_{H})T_{n}(h_{H})T_{n}(g_{H})T_{n}(h_{H})\right]\\\
		&=-2\left(\mathrm{Tr}\left[T_{n}(g_{H})T_{n}(h_{H})\right]-n\right)
		+\left(\mathrm{Tr}\left[T_{n}(g_{H})T_{n}(h_{H})T_{n}(g_{H})T_{n}(h_{H})\right]-n\right).
	\end{align*}
	Then the conclusion follows from Theorem~3.1 of \cite{Takabatake-2022-TraceApp}. This completes the proof. 
\end{proof}
\subsection{Proof of $(\ref{lemma_Whittle_app1})$ in the case $j=0$}
Thanks to Theorems~4.1 and 5.2 of \cite{Adenstedt-1974}, it suffices to prove
\begin{equation*}
	\left\langle\mathbf{1}_{n},T_{n}(h_{H})\mathbf{1}_{n}\right\rangle_{\mathbb{R}^{n}}
	=\left\langle\mathbf{1}_{n},T_{n}(g_{H})^{-1}\mathbf{1}_{n}\right\rangle_{\mathbb{R}^{n}}
	+o(n^{2(1-H)+\epsilon})\ \ \mbox{as $n\to\infty$}
\end{equation*}
for any $\epsilon>0$. First  
we can show
\begin{align*}
	&\left|\left\langle\mathbf{1}_{n},T_{n}(h_{H})\mathbf{1}_{n}\right\rangle_{\mathbb{R}^{n}}
	-\left\langle\mathbf{1}_{n},T_{n}(g_{H})^{-1}\mathbf{1}_{n}\right\rangle_{\mathbb{R}^{n}}\right|\\
	&=\left|\left\langle\mathbf{1}_{n},\left(T_{n}(g_{H})^{-1}-T_{n}(h_{H})\right)\mathbf{1}_{n}\right\rangle_{\mathbb{R}^{n}}\right|\\
	&=\left|\left\langle T_{n}(g_{H})^{-\frac{1}{2}}\mathbf{1}_{n},\left(I_{n}-T_{n}(g_{H})^{\frac{1}{2}}T_{n}(h_{H})T_{n}(g_{H})^{\frac{1}{2}}\right)T_{n}(g_{H})^{-\frac{1}{2}}\mathbf{1}_{n}\right\rangle_{\mathbb{R}^{n}}\right|\\
	&\leq\left\|T_{n}(g_{H})^{-\frac{1}{2}}\mathbf{1}_{n}\right\|_{\mathbb{R}^{n}}
	\left\|\left(I_{n}-T_{n}(g_{H})^{\frac{1}{2}}T_{n}(h_{H})T_{n}(g_{H})^{\frac{1}{2}}\right)T_{n}(g_{H})^{-\frac{1}{2}}\mathbf{1}_{n}\right\|_{\mathbb{R}^{n}}\\
	&\leq\left\|T_{n}(g_{H})^{-\frac{1}{2}}\mathbf{1}_{n}\right\|_{\mathbb{R}^{n}}^{2}
	\left\|I_{n}-T_{n}(g_{H})^{\frac{1}{2}}T_{n}(h_{H})T_{n}(g_{H})^{\frac{1}{2}}\right\|_{F}.
\end{align*}
Then the conclusion 
follows from Lemma~\ref{ext_lemma5.2_D89} and Theorems~4.1 and 5.2 of \cite{Adenstedt-1974}. This completes the proof.
\subsection{Proof of $(\ref{lemma_Whittle_app1})$ in the case $j=1$}
First we can show
\begin{align*}
	\left|\left\langle\mathbf{1}_{n},T_{n}(h_{H}^{(1)})\mathbf{1}_{n}\right\rangle_{\mathbb{R}^{n}}\right|
	&\leq\left\langle\mathbf{1}_{n},T_{n}(|h_{H}^{(1)}|)\mathbf{1}_{n}\right\rangle_{\mathbb{R}^{n}}\\
	&=\left\langle T_{n}(h_{H})^{\frac{1}{2}}\mathbf{1}_{n},\left(T_{n}(h_{H})^{-\frac{1}{2}}T_{n}(|h_{H}^{(1)}|)T_{n}(h_{H})^{-\frac{1}{2}}\right)T_{n}(h_{H})^{\frac{1}{2}}\mathbf{1}_{n}\right\rangle_{\mathbb{R}^{n}}\\
	&\leq\left\|T_{n}(h_{H})^{\frac{1}{2}}\mathbf{1}_{n}\right\|_{\mathbb{R}^{n}}^{2}
		\left\|T_{n}(|h_{H}^{(1)}|)^{\frac{1}{2}}T_{n}(h_{H})^{-\frac{1}{2}}\right\|_{\mathrm{op}}^{2}.
\end{align*}
Then the conclusion 
follows from $(\ref{lemma_Whittle_app1})$ in the case $j=0$ and Lemma~2 in the full version of \cite{Lieberman-Rosemarin-Rousseau-2012}. 
This completes the proof.
\subsection{Proof of $(\ref{lemma_Whittle_app2})$}
First we introduce notation used in the proof. For $j=0,1$, set $g_{H}^{(j)}(\lambda):=(\partial/\partial{H})^{j}g_{H}(\lambda)$, 
\begin{align*}
	&C_{n}^{(j)}(H):=T_{n}(g_{H})^{-1}T_{n}(g_{H}^{(j)})T_{n}(g_{H})^{-1},\\
	&D_{n}^{(j)}(H):=T_{n}(g_{H})^{\frac{1}{2}}(T_{n}(h_{H}^{(j)})-C_{n}^{(j)}(H))T_{n}(g_{H})^{\frac{1}{2}},\\
	&\widetilde{C}_{n}^{(j)}(H):=T_{n}(g_{H})^{\frac{1}{2}}C_{n}^{(j)}(H)T_{n}(g_{H})^{\frac{1}{2}},\\
	&F_{n}(H):=\left\|D_{n}^{(1)}(H)\right\|_{F}^{2}=\mathrm{Tr}\left[\left\{\left(T_{n}(g_{H})T_{n}(h_{H}^{(1)})T_{n}(g_{H})-T_{n}(g_{H}^{(1)})\right)T_{n}(g_{H})^{-1}\right\}^{2}\right],\\
	&\widetilde{F}_{n}(H):=\mathrm{Tr}\left[\left\{\left(T_{n}(g_{H})T_{n}(h_{H}^{(1)})T_{n}(g_{H})-T_{n}(g_{H}^{(1)})\right)T_{n}(h_{H})\right\}^{2}\right].
\end{align*}
Note that $(\ref{lemma_Whittle_app2})$ follows once we have proven that 
	\begin{align}
		&\left\langle\mathbf{1}_{n},T_{n}(h_{H}^{(j)})T_{n}(g_{H})T_{n}(h_{H}^{(j)})\mathbf{1}_{n}\right\rangle_{\mathbb{R}^{n}} \nonumber\\
		&=\left\langle\mathbf{1}_{n},C_{n}^{(j)}(H)T_{n}(g_{H})C_{n}^{(j)}(H)\mathbf{1}_{n}\right\rangle_{\mathbb{R}^{n}}
		+o(n^{2(1-H)+\epsilon})\ \ \mbox{as $n\to\infty$} \label{Whittle_app3_key1}
	\end{align}
	holds for any $\epsilon>0$ and each $j=0,1$ because we can show
	\begin{equation*}
		\left\langle\mathbf{1}_{n},C_{n}^{(j)}(H)T_{n}(g_{H})C_{n}^{(j)}(H)\mathbf{1}_{n}\right\rangle_{\mathbb{R}^{n}}
		=o(n^{2(1-H)+\epsilon})\ \ \mbox{as $n\to\infty$}
	\end{equation*}
	for any $\epsilon>0$ and $j=0,1$ using 
	Theorem~5.2 of \cite{Adenstedt-1974} and Lemma~2 in the full version of \cite{Lieberman-Rosemarin-Rousseau-2012} 
	in the similar way to the proof of Lemma~5.4~(d) of \cite{Dahlhaus-1989}. 
	In the rest of the proof, we will prove (\ref{Whittle_app3_key1}). 
	First we can write
	\begin{align*}
		&\left\langle\mathbf{1}_{n},T_{n}(h_{H}^{(j)})T_{n}(g_{H})T_{n}(h_{H}^{(j)})\mathbf{1}_{n}\right\rangle_{\mathbb{R}^{n}}
		-\left\langle\mathbf{1}_{n},C_{n}^{(j)}(H)T_{n}(g_{H})C_{n}^{(j)}(H)\mathbf{1}_{n}\right\rangle_{\mathbb{R}^{n}}\\
		&=\left\langle\mathbf{1}_{n},\left(T_{n}(h_{H}^{(j)})-C_{n}^{(j)}(H)\right)T_{n}(g_{H})T_{n}(h_{H}^{(j)})\mathbf{1}_{n}\right\rangle_{\mathbb{R}^{n}}\\
		&\quad+\left\langle\mathbf{1}_{n},C_{n}^{(j)}(H)T_{n}(g_{H})\left(T_{n}(h_{H}^{(j)})-C_{n}^{(j)}(H)\right)\mathbf{1}_{n}\right\rangle_{\mathbb{R}^{n}}\\
		&=\left\langle\mathbf{1}_{n},\left(T_{n}(h_{H}^{(j)})-C_{n}^{(j)}(H)\right)T_{n}(g_{H})\left(T_{n}(h_{H}^{(j)})-C_{n}^{(j)}(H)\right)\mathbf{1}_{n}\right\rangle_{\mathbb{R}^{n}}\\
		&\quad+2\left\langle\mathbf{1}_{n},\left(T_{n}(h_{H}^{(j)})-C_{n}^{(j)}(H)\right)T_{n}(g_{H})C_{n}^{(j)}(H)\mathbf{1}_{n}\right\rangle_{\mathbb{R}^{n}}\\
		&=\left\langle T_{n}(g_{H})^{-\frac{1}{2}}\mathbf{1}_{n},D_{n}^{(j)}(H)^{2}T_{n}(g_{H})^{-\frac{1}{2}}\mathbf{1}_{n}\right\rangle_{\mathbb{R}^{n}}
		+2\left\langle T_{n}(g_{H})^{-\frac{1}{2}}\mathbf{1}_{n},D_{n}^{(j)}(H)\widetilde{C}_{n}^{(j)}(H)T_{n}(g_{H})^{-\frac{1}{2}}\mathbf{1}_{n}\right\rangle_{\mathbb{R}^{n}}
	\end{align*}
	so that we obtain
	\begin{align}
		&\left|\left\langle\mathbf{1}_{n},T_{n}(h_{H}^{(j)})T_{n}(g_{H})T_{n}(h_{H}^{(j)})\mathbf{1}_{n}\right\rangle_{\mathbb{R}^{n}}
		-\left\langle\mathbf{1}_{n},C_{n}^{(j)}(H)T_{n}(g_{H})C_{n}^{(j)}(H)\mathbf{1}_{n}\right\rangle_{\mathbb{R}^{n}}\right| \nonumber\\
		&\leq\left\|T_{n}(g_{H})^{-\frac{1}{2}}\mathbf{1}_{n}\right\|_{\mathbb{R}^{n}}^{2}
		\left\|D_{n}^{(j)}(H)\right\|_{\mathrm{op}}
		\left(\left\|D_{n}^{(j)}(H)\right\|_{\mathrm{op}}+2\left\|\widetilde{C}_{n}^{(j)}(H)\right\|_{\mathrm{op}}\right). \label{app2_bound1}
	\end{align}	
	Since we have
	\begin{align*}
		&\left\|D_{n}^{(j)}(H)\right\|_{\mathrm{op}}\leq\left\|D_{n}^{(j)}(H)\right\|_{F},\ \ \left\|D_{n}^{(0)}(H)\right\|_{F}=\left\|I_{n}-T_{n}(g_{H})^{\frac{1}{2}}T_{n}(h_{H})T_{n}(g_{H})^{\frac{1}{2}}\right\|_{F},\\ 
		&\left\|\widetilde{C}_{n}^{(0)}(H)\right\|_{\mathrm{op}}=1,\ \ \left\|\widetilde{C}_{n}^{(1)}(H)\right\|_{\mathrm{op}}\leq\left\|T_{n}(g_{H})^{-\frac{1}{2}}T_{n}(|g_{H}^{(1)}|)^{\frac{1}{2}}\right\|_{\mathrm{op}}^{2},
	\end{align*}
	(\ref{Whittle_app3_key1}) follows from (\ref{app2_bound1}), Lemma~$\ref{ext_lemma5.2_D89}$, Theorem~5.2 of \cite{Adenstedt-1974} and Lemma~2 in the full version of \cite{Lieberman-Rosemarin-Rousseau-2012} once we have proven
	\begin{equation}\label{suff_cond1_Whittle_app3}
		{F}_{n}(H)=o(n^{\epsilon})\ \ \mbox{as $n\to\infty$}
	\end{equation}
	for any $\epsilon>0$. Moreover, (\ref{suff_cond1_Whittle_app3}) follows once we have proven
	\begin{equation}\label{En}
		E_{n}(H):=\left|F_{n}(H)-\widetilde{F}_{n}(H)\right|
		=o(n^{\epsilon})\ \ \mbox{as $n\to\infty$}
	\end{equation}
	for any $\epsilon>0$ because Theorem~3.1 of \cite{Takabatake-2022-TraceApp} gives
	\begin{equation}\label{F_tilde}
		\widetilde{F}_{n}(H)=o(n^{\epsilon})\ \ \mbox{as $n\to\infty$}
	\end{equation}
	for any $\epsilon>0$. 
	Indeed, $\widetilde{F}_{n}(H)$ can be decomposed as the following three terms:
	\begin{align*}
		\widetilde{F}_{n}(H)&=\left(\mathrm{Tr}\left[\left(T_{n}(g_{H})T_{n}(h_{H}^{(1)})T_{n}(g_{H})T_{n}(h_{H})\right)^{2}\right]-nI(H)\right)\\
		&\quad-2\left(\mathrm{Tr}\left[T_{n}(g_{H})T_{n}(h_{H}^{(1)})T_{n}(g_{H})T_{n}(h_{H})\cdot T_{n}(g_{H}^{(1)})T_{n}(h_{H})\right]-nI(H)\right)\\
		&\quad+\left(\mathrm{Tr}\left[\left(T_{n}(g_{H}^{(1)})T_{n}(h_{H})\right)^{2}\right]-nI(H)\right),
	\end{align*}
	where
	\begin{equation*}
		I(H):
		=(2\pi)^{-1}\int_{-\pi}^{\pi}\left|\frac{g_{H}^{(1)}(\lambda)}{g_{H}(\lambda)}\right|^{2}\,\mathrm{d}\lambda.
	\end{equation*}
	Therefore (\ref{F_tilde}) follows from Theorem~3.1 of \cite{Takabatake-2022-TraceApp} since $g_{H}(\lambda)^{2}h_{H}^{(1)}(\lambda)h_{H}(\lambda)=(2\pi)^{-4}g_{H}^{(1)}(\lambda)/g_{H}(\lambda)$.

	In the rest of the proof, we will prove (\ref{En}). 
	First we can bound $E_{n}(H)$ as follows:
	\begin{equation*}
		E_{n}(H)\leq E_{n}^{1}(H)+2E_{n}^{2}(H)+E_{n}^{3}(H),
	\end{equation*}
	where
	\begin{align*}
		E_{n}^{1}(H)&:=\left|\mathrm{Tr}\left[\left(T_{n}(g_{H})T_{n}(h_{H}^{(1)})\right)^{2}\right]
		-\mathrm{Tr}\left[\left(T_{n}(g_{H})T_{n}(h_{H}^{(1)})T_{n}(g_{H})T_{n}(h_{H})\right)^{2}\right]\right|,\\
		E_{n}^{2}(H)&:=\left|\mathrm{Tr}\left[T_{n}(h_{H}^{(1)})T_{n}(g_{H}^{(1)})\right]
		-\mathrm{Tr}\left[T_{n}(g_{H})T_{n}(h_{H}^{(1)})T_{n}(g_{H})T_{n}(h_{H})T_{n}(g_{H}^{(1)})T_{n}(h_{H})\right]\right|,\\
		E_{n}^{3}(H)&:=\left|\mathrm{Tr}\left[\left(T_{n}(g_{H}^{(1)})T_{n}(g_{H})^{-1}\right)^{2}\right]
		-\mathrm{Tr}\left[\left(T_{n}(g_{H}^{(1)})T_{n}(h_{H})\right)^{2}\right]\right|.
	\end{align*}
	Note that we can easily prove $E_{n}^{1}(H)=o(n^{\epsilon})$ and $E_{n}^{2}(H)=o(n^{\epsilon})$ as $n\to\infty$ for any $\epsilon>0$ 
	in the similar way to the proof of (\ref{F_tilde}).  
	Moreover, we can also prove $E_{n}^{3}(H)=o(n^{\epsilon})$ as $n\to\infty$ for any $\epsilon>0$ in the similar way to the proof of Lemma~4 in the full version of \cite{Lieberman-Rosemarin-Rousseau-2012}. 
	Therefore we finish the proof.
\bibliographystyle{acmtrans-ims}
\bibliography{myref_tkbtk}

\end{document}